\documentclass[a4paper,12pt,twoside]{article}

\usepackage{amsmath,amsthm,amscd,amsfonts,amssymb,amsxtra}
\usepackage[hmargin=1.2in,vmargin=1.2in]{geometry}
\usepackage[utf8]{inputenc}
\usepackage{graphicx}
\usepackage{microtype}
\usepackage[pdftex,bookmarks,colorlinks=false]{hyperref}
\usepackage{booktabs,enumitem,caption,subcaption}
\usepackage[mathscr]{euscript}

\usepackage{tikz,calc}
\usepackage{caption,subcaption}
\usetikzlibrary{arrows,backgrounds,automata}
\usepackage{tkz-tab}
\usetikzlibrary{decorations.markings}
\tikzstyle{vertex}=[circle, draw, inner sep=2pt, minimum size=6pt]
\newcommand{\vertex}{\node[vertex]}

\usepackage[auth-sc-lg,affil-it]{authblk}
\newcommand{\E}{\mu}
\newcommand{\C}{\mathcal{C}}
\newcommand{\J}{\mathscr{J}}
\newcommand{\V}{\sigma^2}
\newcommand{\jE}{\E_{J}}
\newcommand{\jV}{\V_{J}}

\newtheorem{theorem}{Theorem}[section]
\newtheorem{corollary}[theorem]{Corollary}
\newtheorem{defnition}[theorem]{Definition}
\newtheorem{proposition}[theorem]{Proposition}

\def\ni{\noindent}

\title{\sc On Certain $J$-Colouring Parameters of Graphs}

\author{Sudev Naduvath}
\affil{\small Centre for Studies in Discrete Mathematics\\ Vidya Academy of Science \& Technology \\ Thrissur - 680501, Kerala, India.\\ {\tt sudevnk@gmail.com}}

\date{}

\begin{document}
\maketitle

\begin{abstract}
In this paper, a new type of colouring called $J$-colouring is introduced. This colouring concept is motivated by the newly introduced invariant called the rainbow neighbourhood number of a graph. The study ponders on maximal colouring opposed to minimum colouring. An upperbound for a connected graph is presented and a number of explicit results are presented for cycles, complete graphs, wheel graphs and for a complete $l$-partite graph.
\end{abstract}

\ni {\small \bf Key Words}: Graph Colouring; $J$-colouring, $J$-colouringnumber, $J$-colouring mean, $J$ colouring variance.  

\vspace{0.2cm}

\ni {\small \bf Mathematics Subject Classification}: 05C15, 05C38, 62A01.

\section{Introduction}

For all  terms and definitions, not defined specifically in this paper, we refer to \cite{BM1, FH,DBW} and for the terminology of graph colouring, see \cite{CZ1,JT1,MK1}. Unless mentioned otherwise, all graphs considered in this paper are simple, finite, connected and non-trivial.

\vspace{0.2cm}

A \textit{graph colouring} is an assignment of colours, labels or weights to the elements of a graph under consideration. Unless stated specifically otherwise, graph colouring is considered to be  an assignment of colours to the vertices of a graph subject to certain conditions. In a \textit{proper colouring} of a graph, its vertices are coloured in such a way that no two adjacent vertices in that graph have the same colour. 

\vspace{0.2cm}

More specifically, a colouring of a graph can be defined as a function $c:V(G)\to {\cal C}=\{c_1,c_2,\ldots,c_r\}$, a set of colours. Note that every proper colouring $c$ of a graph $G$ is a surjective map. 

\vspace{0.2cm}

Several significant researches have been done on graph colouring problems after its introduction in the nineteenth century. Several types of graph colourings and related parameters have been introduced and studied in detail. The most popular and important parameter among them is the \textit{chromatic number} of graphs which is the minimum number of colours required in a proper colouring of the given graph. The concept of chromatic number has been extended to almost all types of graph colourings. Many practical and real life scenarios have been instrumental in the development of different types of graph colouring problems. 

\vspace{0.2cm}

In this paper, we introduce and study certain properties of a particular type of graph colouring, namely \textit{Johan colouring} and related parameters. 

\section{$J$- Colouring of a Graph}

The notion of a rainbow neighbourhood of a graph $G$ with a chromatic colouring $\C$ has been defined in \cite{KSJ1} as the closed neighbourhood $N[v]$ of a vertex $v \in V(G)$ which contains at least one coloured vertex of each colour in the chromatic colouring $\C$ of $G$. Motivated by this study, a new graph colouring, namely Johan colouring \footnote{Named after Dr. Johan Kok, a research colleague of the author.} is introduced as follows.

\begin{defnition}{\rm 
A proper $k$-colouring $\C$ of a graph $G$ is called the \textit{Johan colouring} or the \textit{$J$-colouring} of $G$ if $\C$ is the maximal colouring such that every vertex of $G$ belongs to a rainbow neighbourhood of $G$. A graph $G$ is \textit{$J$-colourable} if it admits a $J$-colouring.}
\end{defnition}

\ni The $J$-colouring of a graph is illustrated in Figure \ref{fig:j-col-1}.

\begin{figure}[h!]
\centering
\begin{tikzpicture}[auto,node distance=2cm, thick,main node/.style={circle,draw,font=\sffamily\Large\bfseries}]
\vertex (v1) at (0:3) [label=right:$v_{1}$]{$c_1$};
\vertex (v2) at (315:3) [label=right:$v_{2}$]{$c_2$};
\vertex (v3) at (270:3) [label=below:$v_{3}$]{$c_4$};
\vertex (v4) at (225:3) [label=left:$v_{4}$]{$c_1$};
\vertex (v5) at (180:3) [label=left:$v_{5}$]{$c_3$};
\vertex (v6) at (135:3) [label=left:$v_{6}$]{$c_4$};
\vertex (v7) at (90:3) [label=above:$v_{7}$]{$c_2$};
\vertex (v8) at (45:3) [label=right:$v_{8}$]{$c_3$};

\path 
(v1) edge (v2)
(v2) edge (v3)
(v3) edge (v4)
(v4) edge (v5)
(v5) edge (v6)
(v6) edge (v7)
(v7) edge (v8)
(v8) edge (v1)
(v1) edge (v6)
(v2) edge (v5)
(v3) edge (v8)
(v4) edge (v7)
;
\end{tikzpicture}
\caption{$J$-Colouring of a graph.}\label{fig:j-col-1}
\end{figure}
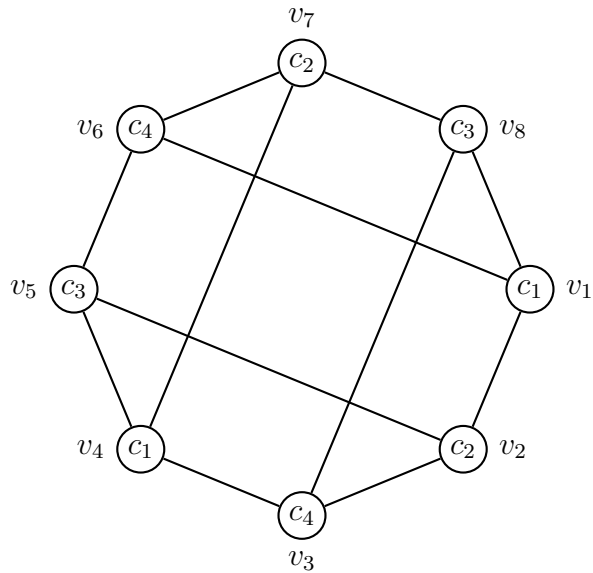

\begin{defnition}{\rm 
A proper $k$-colouring $\C$ of a graph $G$ is called the \textit{modified Johan colouring} or the \textit{$J^*$-colouring} of $G$ if $\C$ is the maximal colouring such that every internal vertex of $G$ belongs to a rainbow neighbourhood of $G$. A graph $G$ is \textit{$J^*$-colourable} if it admits a $J^*$-colouring.}
\end{defnition}

\begin{defnition}{\rm 
The \textit{$J$-colouring number} of a graph $G$, denoted by $\J(G)$, is the maximum number of colours in a $J$-colouring of $G$. Similarly, the \textit{$J^*$-colouring number} of a graph $G$, denoted by $\J^*(G)$, is the maximum number of colours in a $J^*$-colouring of $G$.  }
\end{defnition}

The following result provides an upper bound for the $J$-colouring number of a graph $G$. 

\begin{proposition}\label{Prop-2.1}
Let $G$ be a connected graph. Then, $\J(G)\le \delta(G)+1$.
\end{proposition}
\begin{proof}
Let $c$ be the colouring function defined on $G$ which defines a $J$-colouring $\cal C$ on $G$. Then, all vertices of $G$ is included in a rainbow neighbourhood of $G$ and hence $c(N[v])={\cal C}$ for all $v\in V(G)$.   
\end{proof}

Figure \ref{fig:j-col-1} is an example for a graph which admits a $J$-coloring such that $\J(G)=\delta(G)+1$, the first graph in Figure \ref{fig:j-col-2} is a graph which admits a $J$-coloring such that $\J(G)<\delta(G)+1$ and the second graph in Figure \ref{fig:j-col-2} illustrates a graph with $\J(G)=\delta(G)$.

\begin{figure}[h!]
\begin{center}
\begin{tikzpicture}[scale=0.85] 
\vertex (v1) at (0:3) [label=right:$v_{1}$]{$c_1$};
\vertex (v2) at (315:3) [label=right:$v_{2}$]{$c_2$};
\vertex (v3) at (270:3) [label=below:$v_{3}$]{$c_1$};
\vertex (v4) at (225:3) [label=left:$v_{4}$]{$c_2$};
\vertex (v5) at (180:3) [label=left:$v_{5}$]{$c_1$};
\vertex (v6) at (135:3) [label=left:$v_{6}$]{$c_2$};
\vertex (v7) at (90:3) [label=above:$v_{7}$]{$c_1$};
\vertex (v8) at (45:3) [label=right:$v_{8}$]{$c_2$};

\path 
(v1) edge (v4)
(v1) edge (v6)
(v2) edge (v5)
(v2) edge (v7)
(v3) edge (v6)
(v3) edge (v8)
(v4) edge (v7)
(v5) edge (v8)
(v1) edge (v2)
(v2) edge (v3)
(v3) edge (v4)
(v4) edge (v5)
(v5) edge (v6)
(v6) edge (v7)
(v7) edge (v8)
(v8) edge (v1)
(v1) edge (v6)
(v2) edge (v5)
(v3) edge (v8)
(v4) edge (v7)
;
\end{tikzpicture}
\qquad 
\begin{tikzpicture}[scale=0.85] 
\vertex (v1) at (0:3) [label=right:$v_{1}$]{$c_1$};
\vertex (v2) at (315:3) [label=right:$v_{2}$]{$c_2$};
\vertex (v3) at (270:3) [label=below:$v_{3}$]{$c_3$};
\vertex (v4) at (225:3) [label=left:$v_{4}$]{$c_34$};
\vertex (v5) at (180:3) [label=left:$v_{5}$]{$c_1$};
\vertex (v6) at (135:3) [label=left:$v_{6}$]{$c_2$};
\vertex (v7) at (90:3) [label=above:$v_{7}$]{$c_3$};
\vertex (v8) at (45:3) [label=right:$v_{8}$]{$c_4$};

\path 
(v1) edge (v2)
(v2) edge (v3)
(v3) edge (v4)
(v4) edge (v5)
(v5) edge (v6)
(v6) edge (v7)
(v7) edge (v8)
(v8) edge (v1)
(v1) edge (v3)
(v1) edge (v7)
(v2) edge (v4)
(v2) edge (v8)
(v3) edge (v5)
(v4) edge (v6)
(v5) edge (v7)
(v6) edge (v8)
;
\end{tikzpicture}
\end{center}
\caption{}\label{fig:j-col-2}
\end{figure}

\begin{proposition}\label{Prop-2.2}
For a graph $G$ admitting a $J$-colouring as well as a $J^*$-colouring, $\J(G)\le \J^*(G)\le \delta(G)+1$.
\end{proposition}
\begin{proof}
If $G$ is a graph without pendant vertices, then $\J(G)=\J^*(G)$. If $G$ has pendant vertices, then by Proposition 2.4, $\J(G)\le 2$. 

If $G$ is a star graph $K_{1,n-1}$, then $2\le \J^*(G)\le n$. Otherwise, let $G'$ the connected subgraph of $G$ obtained by removing the pendant vertices of $G$. Since, $d_{G'}(v)\ge 2,\ \forall\, v\in V(G')$, we have $\J^*(G)\ge 2$. Therefore, we have $\J(G)\le \J^*(G)$.
\end{proof}

The $J$-colouring number of certain basic graph classes are determined in the following results.

\begin{proposition}\label{Prop-2.3}
Let $P_n$ denotes a path on $n$ vertices. Then, $\J(P_n)=2$ and $\J^*(P_n)=3$.
\end{proposition}
\begin{proof}
Since any proper colouring of a path $P_n$ consists of $2$ colours, we have $\J(P_n)\ge 2$. Also, since $P_n$ has two pendant vertices, $\delta(P_n)=1$. Then, by Proposition \ref{Prop-2.1}, we have $\J(P_n)\le 2$. Therefore, $\J(P_n)=2$.

Since all internal vertices of $P_n$ have degree $2$, by Proposition \ref{Prop-2.2}, we have $\J^*(P_n)\le 3$. Now, assign colour $c_1$ to the vertices $v_i$ of $P_n$ if $i\equiv 1({\rm mod}\ 3)$, assign colour $c_2$ to the vertices $v_i$ of $P_n$ if $i\equiv 2({\rm mod}\ 3)$ and assign colour $c_3$ to the vertices $v_i$ of $P_n$ if $i\equiv 0({\rm mod}\ 3)$. Clearly, every internal vertex will be adjacent to one vertex of colour $c_1$ and one vertex of $c_2$. Then, this colouring is a $\J^*$-colouring of $P_n$. Therefore, $\J^*(P_n)=3$.
\end{proof}

\begin{theorem}\label{Thm-2.7}
A cycle $C_n$ is $J$-colourable if and only if $n\equiv 0\,({\rm mod}\ 2)$ or $n\equiv 0\,({\rm mod}\ 3)$.
\end{theorem}
\begin{proof}
Let $C_n$ be a cycle on $n$ vertices. Since $d(v)=2$ for all $v\in V(C_n)$, by Proposition \ref{Prop-2.1}, we have $\J(C_n)\le 3$. Then, we consider the following cases.

\vspace{0.2cm}

\textit{Case-1a:} First assume that $n\equiv 0\,({\rm mod}\ 3)$. Then,  assign colour $c_1$ to the vertices $v_i$ of $C_n$, where $i\equiv 1({\rm mod}\ 3)$, assign colour $c_2$ to the vertices $v_i$ of $C_n$, when $i\equiv 2({\rm mod}\ 3)$ and assign colour $c_3$ to the vertices $v_i$ of $C_n$ if $i\equiv 0({\rm mod}\ 3)$ (see first graph in Figure \ref{fig:j-col-3}). Clearly, this colouring is a $J$-colouring of the cycle $C_n$.

\vspace{0.2cm}

\textit{Case-1b:} Next, assume that $n\equiv 0\,({\rm mod}\ 2)$ and $n\not \equiv 0\,({\rm mod}\ 3)$. It is to be noted that no $3$-colouring exists such that every vertex of $C_n$ is in a rainbow neighbourhood. Then, the colouring in which the vertices are coloured alternately by $c_1$ and $c_2$ will itself be the $J$-colouring of $C_n$ (see the second graph in Figure \ref{fig:j-col-3}).

\vspace{0.2cm}

Now, assume that $n\not\equiv 0\,({\rm mod}\ 2)$ as well as $n\not\equiv 0\,({\rm mod}\ 3)$. Since $n\not\equiv 0\,({\rm mod}\ 2)$, $C_n$ is an odd cycle and hence any proper colouring of $C_n$ consists of at least $3$ colours. By the assumptions, we have $n\equiv \pm 1\,({\rm mod}\ 6)$. Then, 

\vspace{0.2cm}

\textit{Case-2a:} Let $n\equiv 1\,({\rm mod}\ 6)$. It is also equivalent to say that $n\equiv  1\,({\rm mod}\ 3)$. Then, for a positive integer $k$, the cycle $C_n$ has $3k+1$ vertices. Now, for all $i\equiv 1\, ({\rm mod}\ 3)$, colour the vertices $v_i,v_{i+1},v_{i+2}$ respectively by $c_1,c_2$ and $c_3$. But, the vertex $v_n$ must be coloured by $c_2$, as cannot assume colours $c_1$ or $c_3$ (since $v_nv_1,v_nv_{n-1}\in E(C_n)$ and $v_1$ has colour $c_1$ and $v_{n-1}$ has colour $c_3$). Hence, the vertex $v_1$ is not in a rainbow neighbourhood, since it is adjacent to two vertices having the colour $c_2$ and not adjacent to any vertex having colour $c_3$. Therefore, $C_n$ has no $J$-colouring in this context.

\vspace{0.2cm}

\textit{Case-2b:} Let $n\equiv -1\,({\rm mod}\ 6)$. It is also equivalent to say that $n\equiv  5\,({\rm mod}\ 6)$ or $n\equiv  2\,({\rm mod}\ 3)$. Then, for a positive integer $k$, the cycle $C_n$ has $3k+2$ vertices. As mentioned above, for all $i\equiv 1\, ({\rm mod}\ 3)$, colour the vertices $v_i,v_{i+1},v_{i+2}$ respectively by $c_1,c_2$ and $c_3$. But, the vertex $v_n$, irrespective of the possible colours it can take, is adjacent to vertices $v_1$ and $v_{n-1}$ having the colour $c_1$ and is not adjacent to a vertex of colour $c_3$. Therefore, in this case also, $C_n$ has no $J$-colouring.

\vspace{0.2cm}

Therefore, the cycle $C_n$ does not have a $J$-colouring, unless either $n\equiv 0\,({\rm mod}\ 2)$ or $n\equiv 0\,({\rm mod}\ 3)$. This completes the proof.
\end{proof}

The two graphs in Figure \ref{fig:j-col-3} illustrate the graphs mentioned respectively in Case-1a and Case-1b of the above theorem. 

\begin{figure}[h!]
\begin{center}
\begin{tikzpicture}[scale=0.85] 
\vertex (v1) at (0:3) [label=right:$v_{1}$]{$c_1$};
\vertex (v2) at (320:3) [label=right:$v_{2}$]{$c_2$};
\vertex (v3) at (280:3) [label=below:$v_{3}$]{$c_3$};
\vertex (v4) at (240:3) [label=below:$v_{4}$]{$c_1$};
\vertex (v5) at (200:3) [label=left:$v_{5}$]{$c_2$};
\vertex (v6) at (160:3) [label=left:$v_{6}$]{$c_3$};
\vertex (v7) at (120:3) [label=above:$v_{7}$]{$c_1$};
\vertex (v8) at (80:3) [label=above:$v_{8}$]{$c_2$};
\vertex (v9) at (40:3) [label=right:$v_{9}$]{$c_3$};

\path 
(v1) edge (v2)
(v2) edge (v3)
(v3) edge (v4)
(v4) edge (v5)
(v5) edge (v6)
(v6) edge (v7)
(v7) edge (v8)
(v8) edge (v9)
(v9) edge (v1)
;
\end{tikzpicture}
\qquad 
\begin{tikzpicture}[scale=0.85] 
\vertex (v1) at (0:3) [label=right:$v_{1}$]{$c_1$};
\vertex (v2) at (324:3) [label=right:$v_{2}$]{$c_2$};
\vertex (v3) at (288:3) [label=below:$v_{3}$]{$c_1$};
\vertex (v4) at (252:3) [label=below:$v_{4}$]{$c_2$};
\vertex (v5) at (216:3) [label=left:$v_{5}$]{$c_1$};
\vertex (v6) at (180:3) [label=left:$v_{6}$]{$c_2$};
\vertex (v7) at (144:3) [label=left:$v_{7}$]{$c_1$};
\vertex (v8) at (108:3) [label=above:$v_{8}$]{$c_2$};
\vertex (v9) at (72:3) [label=above:$v_{9}$]{$c_1$};
\vertex (v10) at (36:3) [label=right:$v_{10}$]{$c_2$};
\path 
(v1) edge (v2)
(v2) edge (v3)
(v3) edge (v4)
(v4) edge (v5)
(v5) edge (v6)
(v6) edge (v7)
(v7) edge (v8)
(v8) edge (v9)
(v9) edge (v10)
(v10) edge (v1)
;
\end{tikzpicture}
\end{center}
\caption{}\label{fig:j-col-3}
\end{figure}

In view of the above theorem, we have the following result on the $J$-colouring number of cycles.

\begin{corollary}\label{Cor-2.8}
Let $C_n$ be a cycle which admits a $J$-colouring. Then, 
\begin{equation*}
\J(C_n)=
\begin{cases}
3 & \text{if}\ n\equiv 0\,({\rm mod}\ 3)\\
2 & \text{if}\ n\equiv 0\,({\rm mod}\ 2) \text{and}\ n\not\equiv 0\,({\rm mod}\ 3).
\end{cases}
\end{equation*}
\end{corollary}

Next, we note that any proper colouring of a complete graph $K_n$ is also a $J$-colouring of $K_n$ and hence we have the following result straight forward.

\begin{proposition}
For a complete graph $K_n$, we have $\J(K_n)=n$.
\end{proposition}

In view of Theorem \ref{Thm-2.7}, it can be noted that any $J$-colouring of a complete bipartite can have exactly $2$ colours. That is, $\J(K_{m.n})=2$. More generally, we have:

\begin{proposition}
For a complete $l$-partite graph $K_{n_1,n_2,\ldots,n_l}$, we have the $J$-colouring number $\J(K_{n_1,n_2,\ldots,n_l})=l$.
\end{proposition}

A \textit{wheel graph} $W_{n+1}$ is a graph obtained by connecting all vertices of a cycle $C_n$ to an external vertex, called central vertex of $W_{n+1}$. That is, $W_{n+1}=K_1+C_n$. The $J$-colouring number of $W_{n+1}$ is determined in the following theorem.

\begin{theorem}
Let $W_{n+1}=K_1+C_n$ be a wheel graph which admits a $J$-colouring. Then, 
\begin{equation*}
\J(W_{n+1})=
\begin{cases}
4 & \text{if}\ n\equiv 0\,({\rm mod}\ 3)\\
3 & \text{if}\ n\equiv 0\,({\rm mod}\ 2) \text{and}\ n\not\equiv 0\,({\rm mod}\ 3).
\end{cases}
\end{equation*}
\end{theorem}
\begin{proof}
Since the central vertex of $K_1$ is adjacent to every vertex of $C_n$, it must have a colour different from the colours of the vertices of $C_n$. This condition holds for all proper colourings in $W_{n+1}$. Therefore, $\J(W_{n+1})=1+\J(C_n)$, provided $C_n$ admits a $J$-colouring. Hence, the result follows immediately. 
\end{proof}

\section{Certain $J$-Colouring parameters of Graphs}

Note that any vertex colouring of a graph $G$ can be considered as a random experiment. Let $\C = \{c_1,c_2, c_3, \ldots,c_k\}$ be a proper $k$-colouring of $G$ and let $X$ be the random variable (\textit{r.v}) which denotes the number of vertices in $G$ having a particular colour. Since the sum of all weights of colours of $G$ is the order of $G$, the real valued function $f(i)$ defined by 
$$f(i)= 
\begin{cases}
\frac{\theta(c_i)}{|V(G)|}; &  i=1,2,3,\ldots,k\\
0; & \text{elsewhere}
\end{cases}$$ 
is the probability mass function (\textit{p.m.f}) of the \textit{r.v} $X$. If the context is clear, we can say that $f(i)$ is the \textit{p.m.f} of the graph $G$ with respect to the given colouring $\C$. Some studies in this direction can be seen in \cite{KSC1,SCK1,SSCK1,SCSK2}.

In view of this terminology, we can define the mean and variance of a graph corresponding to a $J$-colouring, if exists, of a graph $G$ as follows:

\begin{defnition}{\rm 
Let $\C=\{c_1,c_2, c_3, \ldots,c_k\}$ be a $J$-colouring of a graph $G$. Then, 
\begin{enumerate}\itemsep0mm
\item[(i)] the \textit{$J$-colouring mean} of $G$, denoted by $\jE(G)$, is defined to be $\jE(G) = \sum\limits_{i=1}^{k}i\,f(i)$; and
\item[(ii)] the \textit{$J$-colouring variance} of $G$, denoted by $\jV(G)$ is defined as $\jV(G)=\sum\limits_{i=1}^{n}i^2\,f(i)-(\jE(G))^2$.
\end{enumerate}}\end{defnition} 

It is to be noted that like any other proper colouring, $J$-colouring also assigns distinct colours to distinct vertices of a complete graph $K_n$ and hence the p.m.f of $K_n$ will be of the form $f(i)=\begin{cases}
\frac{1}{n}; & \text{if $i=1,2,\ldots,n$};\\
0; & \text{otherwise}.
\end{cases}$

\ni Hence, we have 

\begin{theorem}
The $J$-colouring of a complete graph $K_n$ follows discrete uniform distribution with mean $\frac{n+1}{2}$ and variance $\frac{n^2-1}{12}$.
\end{theorem}

\ni Now, we discuss the $J$-coloring parameters of paths $P_n$.

\begin{proposition}
Let $P_n$ be a path on $n$ vertices. If $n$ is even, then $\jE(P_n)=\frac{3}{2}$ and $\jV(P_n)= \frac{1}{4}$ and if $n$ is odd, then $\jE(P_n)=\frac{3n-1}{2n}$ and $\jV(P_n)= \frac{n^2-1}{4n^2}$.
\end{proposition}
\begin{proof}
Any $J$-coloring of paths can have two colours colours, say $c_1$ and $c_2$ (see Proposition \ref{Prop-2.3}). Then, $P_n$ contains $\lceil \frac{n}{2}\rceil$ vertices with colour $c_1$ and the remaining $\lfloor \frac{n}{2}\rfloor$ vertices have color $c_2$. Here, we have the following two cases:

\textit{Case-1}: Let $n$ be even. Then, the corresponding p.m.f of the path $P_n$ is given by
 $$f(i)=
\begin{cases}
\frac{1}{2}; & \text{if $i=1,2$};\\
0; & \text{otherwise}.
\end{cases}$$ 

Hence, we have $\jE=(1+2)\cdot \frac{1}{2}=\frac{3}{2}$ and $\jV= \left((1^2+2^2)\cdot \frac{1}{2}\right) - \frac{9}{4}= \frac{1}{4}$.

\vspace{0.25cm}

\textit{Case-2}: Let $n$ be odd. Then, the corresponding p.m.f of the path $P_n$ is given by
$$f(i)=
\begin{cases}
\frac{n+1}{2n}; & \text{if $i=1$};\\
\frac{n-1}{2n}; & \text{if $i=2$};\\
0; & \text{elsewhere}.
\end{cases}$$

Then, $\jE= 1\cdot \frac{n+1}{2n}+2\cdot \frac{n-1}{2n}=\frac{3n-1}{2n}$ and $\jV=\left(1^2\cdot \frac{n+1}{2n}+2^2\cdot \frac{n-1}{2n}\right)-\left(\frac{3n-1}{2n}\right)^2=\frac{n^2-1}{4n^2}$.
\end{proof}

\ni Now, we investigate the $J$-coloring parameters of cycles $C_n$.

\begin{proposition}
For $n\equiv 0\,({\rm mod}\ 3)$, $\jE=\frac{1}{2}$ and $\jV=\frac{53}{12}$ and for $n\equiv 0\,({\rm mod}\ 2),\ n\not\equiv 0\,({\rm mod}\ 3)$, $\jE=\frac{3}{2}$ and $\jV=\frac{1}{4}$.
\end{proposition}
\begin{proof}
By Corollary \ref{Cor-2.8}, if $n\equiv 0\,({\rm mod}\ 3)$, then we have three colours in any $J$-colouring of $C_n$. Therefore, the corresponding p.m.f is given by
$$f(i)=
\begin{cases}
\frac{1}{3}; & \text{if $i=1,2,3$};\\
0; & \text{elsewhere}.
\end{cases}$$

Therefore, we have $\jE= (1+2+3)\cdot \frac{1}{3}=\frac{1}{2}$ and $\jV=(1^2+2^+3^2)\cdot \frac{1}{3}-\left(\frac{1}{2}\right)^2=\frac{53}{12}$.

By Corollary \ref{Cor-2.8}, if $n\equiv 0\,({\rm mod}\ 2)$ and $n\not\equiv 0\,({\rm mod}\ 3)$, then we need two colours in any $J$-colouring of $C_n$. Therefore, the corresponding p.m.f is given by
$$f(i)=
\begin{cases}
\frac{1}{2}; & \text{if $i=1,2$};\\
0; & \text{elsewhere}.
\end{cases}$$

Therefore, we have $\jE= (1+2)\cdot \frac{1}{2}=\frac{3}{2}$ and $\jV=(1^2+2^2)\cdot \frac{1}{2}-\left(\frac{3}{2}\right)^2=\frac{1}{4}$. 
\end{proof}

In a similar way, we can determine the $J$-colouring parameters of the wheel graph as follows.

\begin{proposition}
If $n\equiv 0 ({\rm mod}\ 3)$, then $\jE=\frac{2n+4}{n+1}$ and $\jV=\frac{2n^2+14n}{3(n+1)^2}$ and if $n\equiv 0 ({\rm mod}\ 2),\ n\not\equiv 0 ({\rm mod}\ 3)$, then $\jE=\frac{3n+6}{2(n+1)}$ and $\jV=\frac{10n^2+43n+30}{4(n+1)^2}$.
\end{proposition}
\begin{proof}
Note that any $J$-colouring of $W_{n+1}$ has $4$ colours if $n\equiv 0 ({\rm mod}\ 3)$ and $3$ colours if $n\equiv 0 ({\rm mod}\ 2),\ n\not\equiv 0 ({\rm mod}\ 3)$. 

\textit{Case-1}: If $n\equiv 0 ({\rm mod}\ 3)$, then, as mentioned in the previous theorem, we get the corresponding p.m.f as
$$f(i)=
\begin{cases}
\frac{n}{3(n+1)}; & \text{if $i=1,2,3$};\\
\frac{1}{(n+1)}; & \text{if $i=4$};\\
0; & \text{elsewhere}.
\end{cases}$$
Then, $\jE= (1+2+3)\cdot \frac{n}{3n+3}+4\cdot \frac{1}{n+1}=\frac{2n+4}{n+1}$ and $\jV=(1^2+2^+3^2)\cdot \frac{n}{3n+3}+4^2\cdot\frac{1}{n+1}-\left(\frac{2n+4}{n+1}\right)^2=\frac{2n^2+14n}{3(n+1)^2}$.

\textit{Case-2}: If $n\equiv 0 ({\rm mod}\ 3)$, then, as mentioned in the previous theorem, we get the corresponding p.m.f as
$$f(i)=
\begin{cases}
\frac{n}{2(n+1)}; & \text{if $i=1,2$};\\
\frac{1}{(n+1)}; & \text{if $i=3$};\\
0; & \text{elsewhere}.
\end{cases}$$
Then, $\jE= (1+2)\cdot \frac{n}{2n+2}+3\cdot \frac{1}{n+1}=\frac{3n+6}{2(n+1)}$ and $\jV=(1^2+2^2)\cdot \frac{n}{2(n+1)}+3^2\cdot\frac{1}{n+1}-\left(\frac{3n+6}{2(n+1)}\right)^2=\frac{10n^2+43n+30}{4(n+1)^2}$.

\end{proof}

\section{Conclusion}

In this paper, a new type of graph colouring has been introduced and discussed certain graphs which admit this type of graph colouring. Several graph classes can be examined for the admissibility of $J$-colouring and for determining the $J$-colouring number of the graph. The $J$ colouring number for graph joins, product graphs and power graphs can also be studied. Different types of $J$-colourings such as injective $J$ colouring, equitable graph colouring etc. can also be defined and study the characteristics graphs which admit these types of graph colourings. All these highlight a wide scope for further research.  

\section*{Acknowledgement}

The author of the article would like to dedicate this article to his research colleague Dr. Johan Kok, Director, Licensing Services, City of Tshwane, South Africa, on the occasion of his sixtieth birthday on 14th December 2016.

\end{document}